\documentclass[12pt,a4paper]{amsart}
\usepackage[top=35mm, bottom=35mm, left=30mm, right=30mm]{geometry}
\usepackage{mathptmx}
\usepackage{mathrsfs}
\usepackage{amscd}
\usepackage{verbatim}
\usepackage{url}
\usepackage[all]{xy}
\usepackage{color}
\usepackage[colorlinks=true,citecolor=blue]{hyperref}
\usepackage{tikz}
\usetikzlibrary{arrows}
\usetikzlibrary{patterns}
\usetikzlibrary{lindenmayersystems,backgrounds}

\theoremstyle{plain}

\newtheorem{thm}{Theorem}[section]
\newtheorem{lem}[thm]{Lemma}
\newtheorem{prop}[thm]{Proposition}
\newtheorem{cor}[thm]{Corollary}

\theoremstyle{definition}

\newtheorem{ass}{Assumption}

\begin{document}
\title{The structure of pointwise recurrent expansive homeomorphisms}

\author[E.~Shi]{Enhui Shi}
\address[E. Shi]{Soochow University, Suzhou, Jiangsu 215006, China}
\email{ehshi@suda.edu.cn}

\author[H.~Xu]{Hui Xu}
\address[H. Xu]{CAS Wu Wen-Tsun Key Laboratory of Mathematics, University of Science and
Technology of China, Hefei, Anhui 230026, China}
\email{huixu2734@ustc.edu.cn}

\author[Z.~Yu]{Ziqi Yu}
\address[Z. Yu]{Soochow University, Suzhou, Jiangsu 215006, China}
\email{20204207013@stu.suda.edu.cn}



\maketitle


\begin{abstract}
Let $X$ be a compact metric space and let $f:X\rightarrow X$ be a homeomorphism on $X$. We show that
if $f$ is both pointwise recurrent  and expansive, then the dynamical system $(X, f)$ is topologically conjugate
to a subshift  of some symbolic system. Moreover, if $f$ is pointwise positively recurrent, then the
subshift is semisimple; a counterexample is given to show the necessity of positive recurrence to ensure
the semisimilicity.
\end{abstract}

\pagestyle{myheadings} \markboth{E. Shi, H. Xu, and Z. Yu }{Pointwise recurrent expansive homeomorphisms}

\section{Introduction}

By a {\it dynamical system} (or a {\it system} for short), we mean a pair $(X, f)$
where $X$ is a compact metric space and $f$ is a homeomorphism on $X$. Recurrence is one of the most important
 subject in the study of dynamical systems. We know that periodic points, distal points, and almost periodic points are all
 recurrent points. The structure of pointwise recurrent homeomorphisms has been intensively studied by many authors under
some specified assumptions on the dynamics of $f$ or on the topology of the phase space $X$.
A classical result due to Montgomery says that every pointwise periodic homeomorphism on
a connected manifold is periodic (see \cite{Mo}). Similar results were established for pointwise
recurrent homeomorphisms on some surfaces (see \cite{KP, OT}). However, Glasner and Maon showed
that even if $f$ possesses very strong recurrence, the dynamics of $f$ can still be complicated (see \cite{GM}).
Mai and Ye determined the structure of pointwise recurrent  maps having the pseudo orbit tracing property (see \cite{MY}).
 It is well known that minimal homeomorphisms are pointwise
almost periodic. The structure of minimal distal systems was completely described by Furstenberg
in \cite{Fu}. One may consult \cite{Au} for a detailed introduction to the structure theory of
general minimal systems.
\medskip

Expansivity comes from the study of structural stability in differential dynamical systems,
which is also a kind of chaotic property. Many important systems are known to be expansive, such as Anosov systems and
subshifts of symbolic systems. It is known that the circle and the sphere $\mathbb S^2$
admit no expansive homeomorphisms (see \cite{Ao, Hi}) and every compact orientable surface of positive genus admits an expansive
homeomorphism (see \cite{OR}). Ma\~n\'e showed that if $X$ admits an expansive homeomorphism,
then the topological dimension ${\rm dim}(X)<\infty$ (see \cite{Mane}). One may refer to \cite{Ao} for a systematic
introduction to this property.
\medskip

The following celebrated result is  due to Ma\~n\'e, which
clarifies the structure of minimal expansive homeomorphisms.

\begin{thm}[\cite{Mane}] \label{mane}
Let $X$ be a compact metric space. If $X$ admits a minimal expansive homeomorphism $f$,
then  $(X, f)$ is topologically conjugate to a minimal subshift of some symbolic system.
\end{thm}

The purpose of the paper is to extend Theorem \ref{mane} to the case in which $f$ is both pointwise recurrent and expansive.
Recall that a system $(X, f)$ is {\it semisimple} if $X$ is the disjoint union of minimal sets (i.e. every point of $X$
is almost periodic).
\medskip

The following is the main theorem of the paper.

\begin{thm}\label{mainthm}
Let $X$ be a compact metric space and $f$ be a homeomorphism on $X$. If $f$ is both pointwise  recurrent {\color{blue}(resp. positively recurrent)} and expansive, then the dynamical system $(X, f)$ is topologically conjugate to a  subsystem {\color{blue}(resp. semisimple subsystem)} of some symbolic system.\end{thm}

Here, we give some remarks on the conditions in Theorem \ref{mainthm}. For $x\in X$, by the recurrence of $x$, we know that the
orbit closure $\overline {\{f^n(x):n\in \mathbb Z\}}$ is topologically transitive. However,  we cannot conclude that
$\overline {\{f^n(x):n\in \mathbb Z\}}$ is minimal in general, though every point of which is recurrent. In fact,
there do exist non-minimal topologically transitive systems which
are pointwise recurrent (see \cite{DY, KW}). Even if the $f$ in Theorem \ref{mainthm} is semisimple, we can only get that the orbit closure of every point is totally disconnected from Theorem \ref{mane}, which does not mean ${\rm dim}(X)=0$.
In the last section, a counter example is constructed to show that the positive recurrence is necessary for obtaining the semisimilicity.
\medskip

The following corollary is immediate.

\begin{cor}\label{non rec}
Let $X$ be a connected compact metric space and let $f$ be an expansive homeomorphism on $X$. If $X$ is not a single point, then $f$ has a non-recurrent point.
\end{cor}

The following corollary can be deduced from Theorem \ref{mainthm} and the main results in \cite{MY}.

\begin{cor}\label{finite}
Let $X$ be a compact metric space. If $X$ admits a pointwise  positively recurrent homeomorphism  $f$ which is expansive and has the pseudo orbit tracing property, then $X$ is finite.
\end{cor}

\section{preliminaries}

In this section, we will recall some notions and facts around recurrence and expansivity, which will be used in the proof
of the main theorem.

\subsection{Recurrence}

Let $(X, f)$ be a system.  For $x\in X$, the {\it orbit} of $x$ is the set  $orb(x,f):=\{f^{n}(x):n\in\mathbb{Z}\}$; the $\alpha$-{\it limit set} of $x$ is defined to be the set
$$\alpha(x, f):=\{y\in X: \exists\ 0<n_1<n_2<\cdots\ s.t.\ f^{-n_i}(x)\rightarrow y\}$$
 and the $\omega$-{\it limit set} of $x$ is defined to be the set
 $$\omega(x, f):=\{y\in X: \exists\ 0<n_1<n_2<\cdots\ s.t.\ f^{n_i}(x)\rightarrow y\}.$$
The point $x$ is {\it positively recurrent } if $x$ belongs to its $\omega$-limit set; is {\it negatively recurrent } if $x$ belongs to its $\alpha$-limit set; is {\it recurrent } if it is either positively recurrent or negatively recurrent. If $X=\overline{orb(x,f)}$ for some $x\in X$, then $x$ is called a {\it transitive point} and
$f$ is called {\it topologically transitive}. Clearly, a transitive point is a recurrent point.

\medskip

 A subset $E$ of $X$ is $f$-{\it invariant} (or {\it invariant} for short) if $f(E)=E$; we use $f|_E$ to denote the restriction of $f$ to $E$.
 If $E$ is closed and $f$-invariant, then we call $(E, f|_E)$ a {\it subsystem} of $(X, f)$.  From the definitions, we
 see that $orb(x,f)$ is $f$-invariant; both $\alpha(x, f)$ and $\omega(x, f)$ are closed and $f$-invariant.
 If $E$ is closed, invariant, and contains no proper closed invariant subset,
 then $E$ is called a {\it minimal set}. It is clear that
 $E$ is minimal if and only if for every $x\in E$ the orbit $orb(x,f)$ is dense in $E$.
 A point $x\in X$ is {\it almost periodic} if $\overline{orb(x,f)}$ is minimal.
  By an argument of Zorn's Lemma, we know that
 every subsystem contains a minimal set, and hence contains an almost periodic point.
 If $orb(x,f)$ is finite, then $x$ is called a {\it periodic point}. Periodic points are always almost periodic.
 \medskip

 The following proposition can be deduced immediately from the the invariance of $\alpha(x, f)$ and $\omega(x, f)$.

\begin{prop}\label{minimal point}
If $x$ is an almost periodic point, then it is both positively recurrent and negatively recurrent.
\end{prop}

The following proposition is due to Gottschalk (see \cite[Theorem 1]{Gott}). Although the recurrence in \cite[Theorem 1]{Gott} is positive recurrence, the following proposition is a direct corollary.

\begin{prop}\label{fnrec}
Let $f$ be a homeomorphism of a compact metric space $(X,d)$ and $n\geq 1$. Then
 a point $x\in X$ is recurrent with respect to $f$ if and only if it is recurrent with respect to $f^{n}$.
\end{prop}

The following proposition is due to Katznelson and Weiss (see \cite[Lemma 2.1]{KW}).

\begin{prop}\label{trans}
Let $f$ be a homeomorphism of a compact metric space $X$ of dimension $0$. If $f$ is both pointwise positively recurrent and topologically transitive, then $f$ is minimal.
\end{prop}

\subsection{Expansivity and hyperbolic metrics}
Suppose $f$ is a homeomorphism on a compact metric space $X$ with metric $d$. We say  $f$ is  {\it expansive} if there is  $c>0$ such that
 $\sup_{n\in \mathbb{Z}} d(f^{n}(x) , f^{n}(y))>c,$
 for any distinct points $x,y\in X$; we call $c$  an {\it expansivity constant} for $f$.
The following proposition can be seen in \cite{Ao} and is easy to be checked.

\begin{prop}\label{fnexp}
Let $n$ be a positive integer. Then $f$ is expansive if and only if $f^{n}$ is expansive.
\end{prop}

 For $x\in X$ and $r>0$, let $B_{r}(x,d)$  denote the open ball of radius $r$ centering at $x$ with respect to metric $d$, i.e.
$ B_{r}(x,d)=\{ y\in X: d(x,y)<r\}.$
For $x\in X$ and $\varepsilon>0$, let
\begin{eqnarray*}
& & W_{\varepsilon}^{s}(x,d)=\{y\in X: \  d(f^n(x), f^n(y))\leq\varepsilon, \forall n\geq 0\}, \\
& & W_{\varepsilon}^{u}(x,d)=\{y\in X: \  d(f^{-n}(x), f^{-n}(y))\leq\varepsilon, \forall n\geq 0\}.
\end{eqnarray*}
The sets $W_{\varepsilon}^{s}(x, d)$ and $W_{\varepsilon}^{u}(x, d)$ are called respectively the
{\it local stable set} and the {\it local unstable sets} of scale $\varepsilon$ at $x$.
For $x\in X$, the {\it stable set} $W^{s}(x,d)$ and the {\it unstable set} $W^{u}(x,d)$ are defined by
\begin{eqnarray*}
& &  W^{s}(x,d)=\{y\in X: \lim\limits_{n\rightarrow \infty}d(f^n(x), f^n(y))=0\},\\
& & W^{u}(x,d)=\{y\in X: \lim\limits_{n\rightarrow \infty}d(f^{-n}(x), f^{-n}(y))=0\}.
\end{eqnarray*}

The following proposition is Proposition 2.39 in \cite{Ao}.

\begin{prop}\label{stable set}
If $f$ is expansive with a expansivity constant $c$ and $0<\varepsilon<c$, then
\[W^{s}(x,d)=\bigcup_{n\geq 0}f^{-n}(W_{\varepsilon}^{s}(f^n(x),d),\ \   W^{u}(x,d)=\bigcup_{n\geq 0}f^{n}(W_{\varepsilon}^{u}(f^{-n}(x),d).\]
\end{prop}

\begin{thm}\cite[Theorem 1]{Reddy}\label{hypermetric}
Let $f$ be an expansive homeomorphism on a compact metric space $X$. Then there is a compatible metric $D$ on $X$,
$\gamma>0, 0<\lambda<1$ and $a\geq 1$ such that for any $x\in X$
\begin{itemize}
\item [(i)] if $y\in W_{\gamma}^{s}(x, D)$, then
\[ D(f^n(x), f^{n}(y))\leq a\lambda^{n} D(x,y),\ \ \ \text{for any } n\geq 0;\]
\item [(ii)] if $y\in W_{\gamma}^{u}(x, D)$, then
\[ D(f^{-n}(x), f^{-n}(y))\leq a\lambda^{n} D(x,y),\ \ \ \text{for any } n\geq 0.\]
\end{itemize}
\end{thm}

\subsection{Expansivity and topological dimension}

For $x\in X$ and $\varepsilon>0$, let $\Sigma_{\varepsilon }^{s}(x,d)$ (resp. $\Sigma_{\varepsilon }^{u}(x,d)$) denote the connected component of  $W_{\varepsilon}^{s}(x, d)\cap \overline{B_{\varepsilon}(x, d)}$ (resp. $W_{\varepsilon}^{u}(x, d)\cap \overline{B_{\varepsilon}(x, d)}$) containing $x$.
\medskip

The following two results are established by  Ma\~n\'e in \cite{Mane}.

\begin{thm} \label{mane dim}
Let $X$ be a compact metric space. If $X$ admits a minimal expansive homeomorphism $f$,
then  ${\rm dim}(X)=0$.
\end{thm}

The following lemma is crucial in the proof of Theorem \ref{mane dim}.

\begin{lem}\label{manekeylemma}
Let $f$ be an expansive homeomorphism of a compact metric space $(X,d)$. If $\dim(X)>0$, then there is $\varepsilon_0>0$ such that for any $\varepsilon\in(0,\varepsilon_0)$, there is some point $a\in X$ such that
\[ \Sigma_{\varepsilon}^{s}(a, d)\cap \partial B_{\varepsilon}(a,d)\neq\emptyset, \ \ \ \text{or }\ \ \ \Sigma_{\varepsilon}^{u}(a, d)\cap \partial B_{\varepsilon}(a,d)\neq\emptyset.\]
\end{lem}

\subsection{Distality}

Let $f$ be a homeomorphism of a compact metric space $X$ with metric $d$.
If for any distinct points $x,y\in X$, we have $\inf_{n\in\mathbb{Z}}d(f^{n}(x), f^{n}(y))>0$, then $f$ is said to be {\it distal}.
\medskip

From the definition, we immediately have

\begin{prop}\label{dis per}
If $f$ is pointwise periodic, then $f$ is diatal.
\end{prop}

The following lemma is well known. It also holds for any finitely generated group actions (see \cite{LSXX}).

\begin{prop}\cite[Proposition 2.7.1]{BS}\label{dist+expan}
 If $f$ is both distal and expansive, then $X$ is finite.
\end{prop}

\subsection{Pseudo orbit tracing property }

A sequence of points $\{ x_i: a<i<b\} (-\infty\leq a<b\leq +\infty)$ is called a $\delta$ {\it pseudo orbit} for $f$ if $d(f(x_i), x_{i+1})<\delta$ for each $i\in(a,b-1)$. A sequence $\{x_i: a<i<b\}$ is called to be $\varepsilon$ {\it traced} by $x\in X$ if $d(f^{i}(x),x_i)<\varepsilon$ for each $i\in(a,b)$. We say that $f$ has the {\it pseudo orbit tracing property} if for every $\varepsilon>0$, there is $\delta>0$ such that every $\delta$ pseudo orbit for $f$ can be $\varepsilon$ traced by some point of $X$.
\medskip

The following theorem is only a part of the main theorem in \cite{MY} by Mai and Ye.

\begin{thm} \label{rec potp}
Let $X$ be a compact metric space and $f$ be a minimal homeomorphism on $X$. If $f$ has the pseudo orbit tracing property, then
it is conjugate to a subsystem of some adding machine; in particular, it is equicontinuous.
\end{thm}

\section{Some auxiliary lemmas}
In this section, we prepare some technical lemmas which will be used later.
\medskip

Throughout this section, we
let $f$ be an expansive homeomorphism on a compact metric space $X$.

\begin{lem}\label{fnhyp}
 There is a compatible metric $D$ on $X$ such that for any $A\geq 1$, there exist $\delta>0$ and positive integer $N$ such that for any $x\in X$ and any $y\in W_{\delta, f^{N}}^{s}(x,D)$, we have
\begin{eqnarray*}
& &  D(f^{-N}(x), f^{-N}(y))\geq AD(x,y), {\rm and}\\
& & D(f^{iN}(x), f^{iN}(y))\leq \frac{1}{A^{i}} D(x,y), \forall i\geq 0,
\end{eqnarray*}
where $W_{\delta, f^{N}}^{s}(x,D)=\{y\in X: D(f^{iN}(x),f^{iN}(y))\leq \delta, \forall i\geq 0\} $.
\end{lem}
\begin{proof}
Let the metric $D$, $\gamma$, $\lambda$, and $a$ be as in Theorem \ref{hypermetric}. Take a positive integer $N$ with $a\lambda^{N}<1/A$. Take $\delta>0$ be such that
for any $u, v\in X$  with $D(u, v)\leq \delta$,
\begin{equation}\label{eq0}
 \max_{0\leq n\leq N}D(f^{-n}(u), f^{-n}(v))\leq \gamma.
\end{equation}
For each $n\geq 0$, write $n=kN+r$, where $k\geq 1$ and $-N\leq r< 0$. Then for any $y\in W_{\delta, f^{N}}^{s}(x,D)$,  it follows from (1) that
$$D(f^n(f^{-N}x), f^n(f^{-N}y))=D(f^r(f^{(k-1)N}x), f^r(f^{(k-1)N}y))\leq \gamma,$$
which means
$$f^{-N}(y)\in W^{s}_{\gamma}(f^{-N}(x),D)\ {\rm and}\ y\in W^{s}_{\gamma}(x,D).$$
Thus, for any $y\in W_{\delta, f^{N}}^{s}(x,D)$, we have
$$D(x,y)=D(f^{N}(f^{-N}x), f^{N}(f^{-N}y)) $$
$$\leq a\lambda^{N}D(f^{-N}(x), f^{-N}(y))\leq \frac{1}{A}D(f^{-N}(x), f^{-N}(y)),$$
and for each $i\geq 0$,
$$ D(f^{iN}(x), f^{iN}(y))\leq a\lambda^{iN}D(x,y)\leq \frac{1}{A^{i}} D(x,y).$$
This completes the proof.
\end{proof}

Now we propose the following assumption under which some lemmas are obtained.

\begin{ass}\label{assum} There exist $A>1$, a compatible metric $D$ on $X$, and $\delta>0$,  such that for any $x\in X$ and any $y\in W_{\delta}^{s}(x,D)$,
it holds that
\begin{eqnarray*}
& &  D(f^{-1}(x), f^{-1}(y))\geq A D(x,y), {\rm and}\\
& & D(f^{n}(x),f^{n}(y))\leq \frac{1}{A^{n}}D(x,y), \forall n\geq 0.
\end{eqnarray*}
\end{ass}

Under Assumption \ref{assum}, we have the following lemma.

\begin{lem}\label{stab}
If $y,z\in W_{\delta/2}^{s}(x,D)$, then $z\in W_{\delta}^{s}(y,D)$.
\end{lem}
\begin{proof}
For each $n\geq 0$, we have
\[
D(f^{n}(y), f^{n}(z))\leq D(f^{n}(y), f^{n}(x))+D(f^{n}(x), f^{n}(z))\leq \delta/2+\delta/2=\delta.
\]
So, the conclusion holds.
\end{proof}

The following corollary is direct from Assumption \ref{assum} and Lemma \ref{stab}.
\begin{cor}\label{expand}
If $y,z\in W_{\delta/2}^{s}(x,D)$, then $D(f^{-1}(y), f^{-1}(z))\geq A D(y,z)$.
\end{cor}

\begin{cor}\label{expand2}
If $y,z\in f^{-1}\left(W_{\delta/2}^{s}(x,D)\right)$ and $D(y,z)\leq \frac{\delta}{2}$, then $z\in W_{\delta}^{s}(y,D)$.
\end{cor}
\begin{proof}
For each $n\geq 0$,
\[ D(f^{n+1}(y),f^{n+1}(z))\leq D(f^{n}(fy), f^{n}(x))+D( f^{n}(x),f^{n}(fz))\leq \delta/2+\delta/2=\delta.\]
Thus  $z\in W_{\delta}^{s}(y,D)$, since $D(y,z)\leq \frac{\delta}{2}\leq \delta$.
\end{proof}

Under Assumption \ref{assum} with $A\geq 2$, we strengthen Lemma \ref{manekeylemma} as follows.
\begin{lem} \label{minpt}
If $\Sigma_{\varepsilon_1}^{s}(a, D)\cap \partial B_{\varepsilon_1}(a,D)\neq\emptyset$ for some point $a\in X$ and $\varepsilon_1>0$, then
there is an almost periodic point $x^{*}\in X$  and $\varepsilon_2>0$ such that
\[ \Sigma_{\varepsilon_2}^{s}(x^{*}, D)\cap \partial B_{\varepsilon_2}(x^{*},D)\neq\emptyset .\]
\end{lem}

\begin{proof}
Let $\delta>0$ be as in Assumption \ref{assum}. We may as well assume that $\delta<\varepsilon_1$. Since $\Sigma_{\varepsilon_1}^{s}(a, D)\cap \partial B_{\varepsilon_1}(a,D)\neq\emptyset$, we can take a point $y_1\in \Sigma_{\delta/2}^{s}(a, D)$ with $D(a,y_1)=\frac{\delta}{2}$ by the Boundary Bumping Lemma \cite[Chapter V]{Nadler}.

By Corollary \ref{expand}, we have $D(f^{-1}(a), f^{-1}(y_1))\geq A D(a,y_1)\geq\delta$. Thus there is a point $y_2\in \Sigma_{\delta/2}^{s}(f^{-1}(a), D)$ with $D(f^{-1}(a),y_2)=\frac{\delta}{2}$. Repeating this process, we obtain a sequence $(y_n)$ of points in $X$ satisfying
\[ y_{n+1}\in \Sigma_{\delta/2}^{s}(f^{-n}(a), D)\ \ \ \text{ and }\ \  D(f^{-n}(a), y_{n+1})=\frac{\delta}{2},\]
for any $n\geq 0$.

Since the $\alpha$-limit set $\alpha(a, f)$ of $a$ is a nonempty closed invariant subset of $X$, there is an increasing sequence $n_1<n_2<n_3<\cdots$ such that $f^{-n_i}(a)\rightarrow x^{*}$ with $x^{*}$ being an almost periodic point in  $\alpha(a, f)$. By passing to some subsequence, we may further assume that $\Sigma_{\delta/2}^{s}(f^{-n_i}(a),D)$ converges in the hyperspace (\cite[Chapter IV]{Nadler}) to a compact connected subset $K$ of $X$.

 Notice that for any point $p\in K$, there is a sequence $x_{n_i}\in \Sigma_{\delta/2}^{s}(f^{-n_i}(a), D)$ with $x_{n_i}\rightarrow p$. Then for each $i\geq 1$ and $n\geq 0$, $D(f^n(f^{-n_i}a), f^{n}(x_{n_i}))\leq \delta/2$. Thus $D(f^{n}(x^*), f^{n}(p))\leq \delta/2$ and hence $p\in W_{\delta/2}^{s}(x^*,D)$. It follows from the connectedness of $K$ that $K\subset\Sigma_{\delta/2}^{s}(x^*,D)$. Let $q$ be a limit point of $(y_{n_i+1})$. Then $q\in K$ and $D(x^*, q)=\delta/2$. By taking a positive $\varepsilon_2\leq\delta/2$, we have
\[ \Sigma_{\varepsilon_2}^{s}(x^{*}, D)\cap \partial B_{\varepsilon_2}(x^{*},D)\neq\emptyset .\]
Thus we complete the proof.
\end{proof}

\section{proof of Theorem \ref{mainthm} and Corollary \ref{finite}}

To prove Theorem \ref{mainthm}, we need only to prove ${\rm dim}(X)=0$. If this is true, by a
canonical coding technique, we get that $f$ is conjugate to a subshift of some symbolic system.
Then applying Proposition \ref{trans}, we see that this subshift is semisimple if the recurrence is strengthened to positive recurrence.
 Thus we complete the proof.
\medskip

From Propositions \ref{fnrec} and \ref{fnexp}, we see that $f$ is both pointwise recurrent and expansive if
and only if $f^n$ is  both pointwise recurrent and expansive for any $n>1$. Thus, to prove  ${\rm dim}(X)=0$,
by Lemma \ref{fnhyp} and by replacing $f$ by some $f^N (N>1)$ if necessary, we may assume that $f$ satisfies
Assumption \ref{assum}.

\medskip
Here we should note that the number $A$ in Assumption \ref{assum} can be taken arbitrarily large, since
we can always replace $f$ by some $f^N$ with $N$ being sufficiently large.

\begin{proof}[Proof of Theorem \ref{mainthm}]
By the discussions at the beginning of this section, we need only to show that $\dim(X)=0$;
and we can suppose that  $f$ satisfies Assumption \ref{assum} with $A=7$. Let $\delta>0$
be as in Assumption \ref{assum}.

To the contrary, assume that $\dim(X)>0$. Then by Lemma \ref{manekeylemma}, there is a point $a\in X$ and $\varepsilon_1>0$ such that
\[ \Sigma_{\varepsilon_1}^{s}(a, D)\cap \partial B_{\varepsilon_1}(a,D)\neq\emptyset, \ \ \ \text{or }\ \ \ \Sigma_{\varepsilon_1}^{u}(a, D)\cap \partial B_{\varepsilon_1}(a,D)\neq\emptyset.\]
By replacing $f$ with $f^{-1}$ if necessary, we may as well assume that $\Sigma_{\varepsilon_1}^{s}(a, D)\cap \partial B_{\varepsilon_1}(a,D)\neq\emptyset$. Then it follows from Lemma \ref{minpt} that there is an almost periodic point $x^{*}\in X$ and $0<\varepsilon_2\leq \delta/2$
such that
\[ \Sigma_{\varepsilon_2}^{s}(x^{*}, D)\cap \partial B_{\varepsilon_2}(x^{*},D)\neq\emptyset .\]
Let $\Lambda=\overline{orb(x^*,f)}$, which is totally disconnected by Theorem \ref{mane dim}. Since
 $\Sigma_{\varepsilon_2}^{s}(x^*,D)$ is connected and $\Sigma_{\varepsilon_2}^{s}(x^*,D)\cap \partial
 B(x^*, \varepsilon_2)\neq \emptyset$, $\Lambda\cap \Sigma_{\varepsilon_2}^{s}(x^*,D) $ is a proper closed subset of $\Sigma_{\varepsilon_2}^{s}(x^*,D)$ in the relative topology. Pick $y\in \Sigma_{\varepsilon_2}^{s}(x^*,D)$ and $\gamma>0$ with
\begin{equation}\label{eq3}
 \Sigma_{\varepsilon_2}^{s}(x^*,D)\cap B_{2\gamma}(y, D)\subset \Sigma_{\varepsilon_2}^{s}(x^*,D)\setminus \Lambda.
\end{equation}
Let $\Sigma_0$ be the connected component of $\Sigma_{\varepsilon_2}^{s}(x^*,D)\cap \overline{B_{\gamma}(y, D)}$ containing $y$. Take $z\in \Sigma_{\varepsilon_2}^{s}(x^*,D)\cap \partial B_{\gamma}(y, D) $. Since $\varepsilon_2\leq \delta/2$, $z\in W_{\delta}^{s}(y,D)$ by Lemma \ref{stab}.

Set $y_0=y$ and $z_0=z$. Then by Corollary \ref{expand},
\[ D(f^{-1}(y_0), f^{-1}(z_0))\geq A D(y_0,z_0)=7\gamma.\]
Thus either $D(f^{-1}(y_0), y)>3\gamma$ or $D(f^{-1}(z_0), y)>3\gamma$. Set $y_1=f^{-1}(y_0)$ if $D(f^{-1}(y_0), y)>3\gamma$, otherwise set $y_1=f^{-1}(z_0)$.
Let $\Sigma_1$ be the connected component of $f^{-1}(\Sigma_0)\cap \overline{B_{\gamma}(y_1, D)}$ that contains $y_1$. Take $z_1\in\Sigma_1$ with $D(y_1, z_1)=\gamma$.
Noting that $\gamma\leq \delta/2$, by Corollary \ref{expand2}, we have $z_1\in W_{\delta}^{s}(y_1,D)$. Then  $D(f^{-1}(y_1), f^{-1}(z_1))\geq A D(y_1,z_1)=7\gamma$.
Thus either $D(f^{-1}(y_1), y)>3\gamma$ or $D(f^{-1}(z_1), y)>3\gamma$.  Choose $y_2$ from $\{f^{-1}(y_1), f^{-1}(z_1)\}$ such that $D(y_2, y)>3\gamma$. Let $\Sigma_2$ be the connected component of $f^{-1}(\Sigma_1)\cap \overline{B_{\gamma}(y_2, D)}$ that contains $y_2$. Take $z_2\in\Sigma_2$ with $D(y_2, z_2)=\gamma$. Repeating this process, we obtain a sequence $(\Sigma_i)$ of compact connected subsets such that $\Sigma_i\cap B_{2\gamma}(y,D)=\emptyset$ for each $i\geq 1$ and
\[ \overline{B_{\gamma}(y,D)}\supset \Sigma_0\supset f(\Sigma_1)\supset f^{2}(\Sigma_2)\supset\cdots.\]
Take a point $w\in\bigcap_{i\geq1} f^{i}(\Sigma_i)$. Then $w\in \overline{B_{\gamma}(y,D)}$ and $f^{-i}(w)\notin B_{2\gamma}(y,D)$ for each $i\geq 1$. Thus $w$ is not negatively recurrent. Since $w\in W^{s}(x^*,D)$, we have
\[D(f^{n}(w), \Lambda) \rightarrow 0\ \ \ \text{ as } \ \ n\rightarrow+\infty.  \]
 By the choice of $y$, we see that $w$ is not positively recurrent.  To sum up, $w$ is not recurrent. This is a contradiction.
\end{proof}

\begin{proof}[Proof of Corollary \ref{finite}]
From Theorem \ref{mainthm} and Theorem \ref{rec potp}, we see that
for each $x\in X$, the orbit closure $\overline{orb(x, f)}$ is
both expansive and equicontinuous, which implies that $x$ is periodic.
Thus $f$ is pointwise periodic, and so $X$ is finite by Propositions \ref{dis per}
and \ref{dist+expan}.
\end{proof}

\section{A non-minimal, pointwise recurrent and transitive subshift }
In this section, we will construct a non-minimal, pointwise recurrent and transitive subshift to show that
the pointwise positive recurrence in Proposition \ref{trans} and the second part of Theorem \ref{mainthm} cannot be weaken to pointwise recurrence.

\medskip
We define a sequence $\xi$ as
\[\cdots 0000\overset{\overset{\xi(0)}{\uparrow}}{1}010010100010100101000001010010100010100101\cdots.\]
Precisely, we let $\xi(n)=0$ for all $n<0$ and define $\xi(n)$ for $n\geq 0$ inductively.  Set
\begin{eqnarray*}
&\omega_1=1,\omega_2=\omega_10\omega_1=101, \omega_3=\omega_200\omega_2=10100101,\\
&\cdots, \omega_{n+1}=\omega_n\underset{n \text{ zeros}}{\underbrace{000\cdots000}}\omega_n,\cdots.
\end{eqnarray*}
Denote by $\ell_n$  the length of $\omega_n$ and let
\[\xi(0)\xi(1)\cdots\xi(\ell_n-1)=\omega_n.\]
In such a way, we eventually get the sequence $\xi$. Let $T: \{0,1\}^{\mathbb{Z}}\rightarrow\{0,1\}^{\mathbb{Z}}$ be the left shift. Let $X=\overline{orb(\xi, T)}$.
It is clear that $X$ is non-minimal, since ${\bf 0}=\cdots0000\cdots$ is in $X$ which is a fixed point. In the remaining part, we will show that $(X,T)$ is pointwise recurrent.

For $x\in \{0,1\}^{\mathbb{Z}}$, let $\overline{x}$ be the reflection of $x$ with respect to the origin, i.e. $\overline{x}(n)=x(-n)$.

\medskip
\noindent{\bf Claim 1.}  For any $x\in X$, we have $\overline{x}\in X$.
\begin{proof}
From the definition of $\xi$, we see that $\overline{\xi}=\lim_{n\rightarrow\infty}T^{\ell_n-1}\xi$. Then $\xi=\lim_{n\rightarrow\infty}T^{-(\ell_n-1)}\overline{\xi}$ and hence $\overline{\xi}$ is also a transitive point in $X$. Thus for any $x\in X$, we have $\overline{x}\in X$.
\end{proof}

\medskip
From the construction of $\xi$, we have the following claim.
\medskip

\noindent {\bf Claim 2.} For any $n>m>0$,
\begin{itemize}
\item [(a)] if $\xi(m) \cdots\xi(n)=00\cdots001$, then $\xi(n)\cdots\xi(n+\ell_{n-m}-1)=\omega_{n-m}$;
\item [(b)] if $n-m>\ell_{k}+k$, then  $\xi(m) \cdots\xi(n)$ contains $k-1$ consecutive $0$'s.
\end{itemize}

\medskip
\noindent{\bf Claim 3.} For any $x\in X\setminus\{{\bf 0}\}$, $1$ occurs in $x$ infinitely often.
\begin{proof}
To the contrary, suppose that there are finitely many $1$'s in $x$. Then there is $N>2$ such that $x(n)=0$ for any $|n|\geq N$. We further assume that $2^{N}> 2N+1$. Since $x\neq{\bf 0}$, we may assume that $T^{n_i}\xi\rightarrow x$ for some increasing sequence $0<n_1<n_2<\cdots$. Set $x_i=T^{n_i}\xi$  and take $i$  such that $n_i>2N$ and $x_i$ coincides with $x$ on $[-2N-\ell_{3N},2N+\ell_{3N}]$. Let $k$ be the least integer such that $x(k)=1$. Then
\[x_i(-2N)\cdots x_i(k)=\xi(n_i-2N)\cdots \xi(n_i+k)=00\cdots001.\]
By Claim 2 (a), we have
\[ x(k)\cdots x(k+\ell_{k+2N}-1)=x_i(k)\cdots x_i(k+\ell_{k+2N}-1)=\omega_{k+2N}.\]
Particularly, we have $x(k+\ell_{k+2N}-1)=1$. Note that $\ell_n\geq 2^{n}$ for any $n>2$. Thus $k+\ell_{k+2N}-1\geq 2^{N}-N-1>N$. This is absurd since $x$ takes $0$ outside $[-N,N]$. Thus we have proved the claim.
\end{proof}

\medskip
Now we are ready to show that $(X, T)$ is pointwise recurrent.  It suffices to show every $x\in X\setminus\{{\bf 0}\}$ is recurrent. By Claim 1, $\overline{\xi}\in X$. Noting that the recurrence of $x$ is equivalent to the recurrence of $\overline{x}$, by Claim 1 and Claim 3, we may as well
 assume that there are infinitely many $1$'s in the positive part of $x$.
Let $0<n_1<n_2<\cdots$ be such that $T^{n_i}\xi\rightarrow x$.

\medskip
\noindent{\bf Claim 4.} For any $k>0$, there exist $m(k)>0$ such that
\[ x(m(k))\cdots x(m(k)+k)=00\cdots001.\]

\begin{proof}
Let $r>0$ be such that $r\geq\ell_{k+1}+k+1$ and $x(r)=1$. Let $i>0$ be such that $T^{n_i}\xi$ coincides with $x$ on $[0,r]$.
Thus $x(0)\cdots x(r)=\xi(n_i)\cdots\xi(n_i+r)$. By Claim 2 (b), there are $k$ consecutive $0$'s in $x(0)\cdots x(r)$. Since $x(r)=1$, there is $m(k)\in(0,r)$ such that
\[ x(m(k))\cdots x(m(k)+k)=00\cdots001.\]
\end{proof}

To show the recurrence of $x$, it suffices to show that for any $N>0$, there is $n>0$ such that
\[x(n)\cdots  x(n+2N)=x(-N)\cdots x(N).\]
Since there is some $i$ such that $T^{n_i}\xi$ coincides with $x$ on $[-N,N]$, we have
\[ x(-N)\cdots x(N)=\xi(n_i)\cdots\xi(n_i+2N).\]
Let $s>0$ be such that $\ell_{s}> n_i+2N$. Thus $x(-N)\cdots x(N)$ is a subword of $\omega_{s}$. Set $t=m(n_i+2N)+n_i+2N+\ell_{n_i+2N}$. Let $j$ such that $T^{n_j}\xi$ coincides with $x$ on $[-t,t]$.
By Claim 4,  $x(m(n_i+2N))\cdots x(m(n_i+2N)+n_i+2N)=00\cdots001$.
Thus
\[ (T^{n_j}\xi)(m(n_i+2N))\cdots (T^{n_j}\xi)(m(n_i+2N)+n_i+2N)=00\cdots001.\]
By Claim 2 (b), we have
\begin{eqnarray*}
&&x(m(n_i+2N)+n_i+2N)\cdots x(m(n_i+2N)+n_i+2N+\ell_{n_i+2N}-1)\\
&=& (T^{n_j}\xi)(m(n_i+2N)+n_i+2N)\cdots (T^{n_j}\xi)(m(n_i+2N)+n_i+2N+\ell_{n_i+2N}-1)\\
&=& \xi(n_j+m(n_i+2N)+n_i+2N)\cdots \xi(n_j+m(n_i+2N)+n_i+2N+\ell_{n_i+2N}-1)\\
&=&\omega_{n_i+2N}.
\end{eqnarray*}
Note that $x(-N)\cdots x(N)$ is a subword of $\omega_{s}$. Thus $x(-N)\cdots x(N)$ is a subword of $$x(m(n_i+2N)+n_i+2N)\cdots x(m(n_i+2N)+n_i+2N+\ell_{n_i+2N}-1),$$ which implies the recurrence of $x$.

\subsection*{Acknowledgements}
We would like to thank Professor Bingbing Liang for helpful comments.
The work is supported by NSFC (No. 11771318, 11790274).

\end{document}